\documentclass{amsart}
\pdfoutput=1 
\usepackage{amssymb, latexsym, xfrac}
\usepackage[shortlabels]{enumitem}
\usepackage{hyperref}
\usepackage{pdfsync}

\title[Mapping tori and BNS invariant]{Mapping tori of free group automorphisms, and the Bieri-Neumann-Strebel invariant of graphs of groups}
\author{Christopher H. Cashen} 
\address{
Fakult\"at f\"ur Mathematik, Universit\"at Wien}
\email{\href{mailto:christopher.cashen@univie.ac.at}{christopher.cashen@univie.ac.at}}
\author{Gilbert Levitt}
\address{Laboratoire de Math\'ematiques Nicolas Oresme, Universit\'e de Caen et CNRS (UMR 6139)}
\email{\href{mailto:levitt@unicaen.fr}levitt@unicaen.fr}
\date{December 30, 2014.}

\thanks{First author partially supported by 
 European Research Council (ERC) grant \#259527 and the
Austrian Science Fund (FWF):M1717-N25. 
 Both authors partially supported by Agence Nationale de la Recherche
 (ANR) grant ANR-2010-BLAN-116-01 GGAA. 
They also thank the Erwin Schr\"odinger Institute workshop ``Geometry and Computation in Groups''.}

\hypersetup{
    pdftitle={Bieri-Neumann-Strebel Invariants for Group Splittings},    
    pdfauthor={Christopher H. Cashen and Gilbert Levitt},     
    pdfkeywords={BNS invariant, fibration, free group automorphism, mapping torus}, 
    colorlinks=true,       
    linkcolor=black,          
    citecolor=black,        
    filecolor=black,      
    urlcolor=black           
}

\theoremstyle{plain}
\newtheorem{theorem}{Theorem}[section]
\newtheorem{lemma}{Lemma}[section]
\newtheorem{proposition}{Proposition}[section]
\newtheorem{corollary}{Corollary}[section]
\theoremstyle{remark}
\newtheorem*{rem*}{Remark}
\theoremstyle{definition}
\newtheorem{example}{Example}[section]
\newtheorem{definition}{Definition}[section]
\newtheorem{remark}{Remark}[section]

\def\makeautorefname#1#2{\expandafter\def\csname#1autorefname\endcsname{#2}}
\let\fullref\autoref

\makeautorefname{theorem}{Theorem} 
\makeautorefname{lemma}{Lemma} 
\makeautorefname{proposition}{Proposition} 
\makeautorefname{corollary}{Corollary} 
\makeautorefname{definition}{Definition}
\makeautorefname{section}{Section}
\makeautorefname{subsection}{Section}
\makeautorefname{subsubsection}{Section}
\makeautorefname{remark}{Remark}
\makeautorefname{example}{Example}

\makeatletter 
\let\c@lemma=\c@theorem 
\makeatother
\makeatletter 
\let\c@proposition=\c@theorem 
\makeatother
\makeatletter 
\let\c@corollary=\c@theorem 
\makeatother
\makeatletter 
\let\c@definition=\c@theorem 
\makeatother
\makeatletter 
\let\c@remark=\c@theorem 
\makeatother
\makeatletter 
\let\c@example=\c@theorem 
\makeatother

\DeclareMathOperator{\Aut}{Aut} 
\DeclareMathOperator{\Out}{Out} 
\DeclareMathOperator{\upg}{UPG} 

\DeclareMathOperator{\GL}{GL}

\DeclareMathOperator{\gpcenter}{Z}
\newcommand{\R}{\mathbb{R}}
\newcommand{\Q}{\mathbb{Q}}
\newcommand{\F}{\mathbb{F}}
\newcommand{\Z}{\mathbb{Z}}

\DeclareMathOperator{\Cay}{Cay}

\newcommand{\gbs}{\mathrm{GBS}}

\newcommand{\from}{\colon\thinspace}
\newcommand{\onto}{\twoheadrightarrow}
\newcommand{\into}{\hookrightarrow}

\begin{document}
\begin{abstract}
Let $G$ be the mapping torus of a polynomially growing automorphism of a finitely generated free group. We determine which epimorphisms from $G$ to $\Z$ have finitely generated kernel, and we compute the rank of the kernel. We thus describe all possible ways of expressing $G$ as the mapping torus of a free group automorphism. This is similar to the case for 
3--manifold groups, and different from the case of mapping tori of
exponentially growing free group automorphisms.
The proof uses a hierarchical decomposition of $G$ and requires determining the Bieri-Neumann-Strebel invariant of the fundamental group of certain graphs of groups. 
\end{abstract}

\maketitle

\section{Introduction}
Given an automorphism $\alpha$ of a group $F$, one may form its mapping torus $$G_\alpha=F\rtimes_\alpha\Z=\langle F,t\mid t^{-1}ft=\alpha(f)\rangle$$ and obtain an exact sequence $1\to F\to G_\alpha\to\Z\to1$. Geometrically, any fibration over the circle leads to such an exact sequence, with $F$ the fundamental group of the fiber and $\alpha$ induced by the monodromy. If the fiber is compact, $F$ and $G_\alpha$ are finitely generated.

Conversely, any epimorphism $\varphi:G\to\Z$ yields a split exact sequence $1\to\ker \varphi\to G\to\Z\to1$, but in general $G$ finitely generated does not imply that $\ker\varphi$ is finitely generated. We will say that $\varphi:G\to\Z$ (or any multiple of $\varphi$) comes from a \emph{fibration} if $\ker\varphi$ is finitely generated. We then call $F=\ker\varphi$ the \emph{fiber}, and we define the \emph{monodromy} as the element $\Phi\in\Out(F)$ determined by the action of the generator of $\Z$. The group $G$ is isomorphic to the mapping torus $G_\alpha$ of any   $\alpha\in\Aut(F)$  representing  $\Phi$.

With this terminology, one may ask, given $G$, which maps
$\varphi:G\to\Z$ come from a fibration. This was answered by Thurston
\cite{Thu86} when $G$ is a 3-manifold group, or equivalently when the fiber is
a surface group. 
The question is more subtle when $F$ is a free group. Indeed, the
class of free-by-cyclic groups remains largely mysterious, in spite of
results such as \cite{BriGro10, FeiHan99, GauLus07}.

Recent work by Dowdall, Kapovich, and Leininger \cite{DowKapLei13a,DowKapLei13b} studies mapping tori of exponentially growing automorphisms of free groups. We study the opposite case, when monodromies are polynomially growing automorphisms. 

 \begin{theorem}[see Theorems \ref{main} and  \ref{prop:fiberrank}]\label{mainintro}
Let $G=G_\alpha=\F_n\rtimes_\alpha\Z$ for $\alpha\in\Aut(\F_n)$ polynomially growing, with $n\ge2$. 
There exist elements $t_1,\dots, t_{n-1}$ in $G$ (not necessarily distinct), and $k\ge1$, 
such that, given an epimorphism $\varphi:G\to\Z$:
 \begin{itemize}
 \item If some $\varphi(t_i)$ is $0$, then
    $\ker\varphi$ virtually
   surjects onto $\F_\infty$ (an infinitely generated free group).
  \item If no $\varphi(t_i)$ is $0$, then 
    $\ker\varphi$ is a finitely generated free group whose rank is:
    \[r=1+\frac1k\sum_{i=1}^{n-1} | \varphi(t_i) |. \]
 \end{itemize}
\end{theorem}

    Freeness of $\ker\varphi$ when finitely generated was proved in
    \cite[Remark~2.7]{GeoMihSap01}. 
Any $\varphi$ such that no $\varphi(t_i)$ is 0 expresses $G$
    as the mapping torus of an automorphism of a finitely generated
    free group. 
It is known \cite{Mac02} that this automorphism is polynomially growing, with the same degree of growth as the original $\alpha$.

It follows from the theorem  that any homomorphism $\varphi:G\to\R$ may be approached by fibrations, and $\ker\varphi$ virtually maps onto $\F_\infty$ if it is not finitely generated. We will also see that $G$ cannot be written as the mapping torus of an injective, non surjective, endomorphism of a finitely generated free group. None of these properties holds for mapping tori of arbitrary automorphisms   of free groups. 
    
    We prove \fullref{mainintro} by induction, using the fact that
$G$ admits a hierarchy: it may be iteratively split along cyclic groups until vertex groups are $\Z^2$. The inductive step   requires us   to understand fibrations of fundamental groups of graphs of groups. We do so in the more general context of the \emph{Bieri-Neumann-Strebel (BNS) invariant}. 

Recall \cite{BieNeuStr87,Str12} that  
the BNS invariant 
$\Sigma(G)$ (or $\Sigma^1(G)$, in the notation of \cite{Str12}) of a finitely generated group $G$ is a
certain open subset of the 
sphere $S(G)$ of projective classes $[\varphi]$ of non-zero homomorphisms $\varphi:G\to
\R$. The connection with the previous discussion is that $\varphi:G\onto\Z$ is a fibration if and only if $\Sigma(G)$ contains   $[\varphi]$ and $[-\varphi]$.

We investigate the BNS invariant of graphs of groups in \fullref{sec:graphsofgroups}.
  So far, no systematic such study   has appeared in print; after
  receiving a preliminary version of this paper, Ralph Strebel sent us
  the unpublished monograph \cite{BieriStrebelMonograph}, whose Theorem II.5.1 is very similar to  our results. 
  
Say that a graph of groups is reduced if no edge carries a trivial amalgam  $A*_AB$, and define 
the class $\varsigma$ as consisting of all finitely generated groups $H$ such that $\Sigma(H)=S(H)$ (this condition is equivalent to $[H,H] $ being finitely generated).

\begin{theorem}[Corollaries \ref{varsigmaamalgam} and \ref{graphofvarsigmagroups}] \label{BNSintro}
Let $G$ be the fundamental group of a finite reduced graph of groups $\Gamma$, with $G$ finitely generated.   
 Assume that $\Gamma$ is not an ascending HNN-extension. Consider a nonzero map $\varphi:G\to\R$.
 \begin{itemize}
\item If edge groups are in $\varsigma$,
then  $[\varphi]\in \Sigma(G)$ if and only if $\varphi $ is non-trivial on every edge group and  $[\varphi_{ | G_v}]\in\Sigma(G_v)$ for every vertex group $G_v$. 
\item If vertex groups are in $\varsigma$, then
  $[\varphi]\in\Sigma(G)$ if and only if $\varphi $ is non-trivial on every edge
  group.  
In particular, $\Sigma(G)=-\Sigma(G)$ is the complement of a finite number of rationally defined subspheres. 
\end{itemize}
\end{theorem}

The second assertion applies in particular to graphs of virtually
polycyclic groups. 
In \fullref{sec:gbs} we specialize it to $\gbs_n$ groups, defined as fundamental groups of finite graphs of groups with every vertex and edge group isomorphic to $\Z^n$. If $G$ is a non-solvable $\gbs_n$ group, then  $[\varphi]\in\Sigma(G)$ if and only if $\varphi(H)\ne0$, where $H$ is any edge group. When $G$ is unimodular (i.e.\ virtually $\Z^n\times \F_k$ for some $k\ge2$), or $n=1$, this is equivalent to $\varphi(Z)\ne0$ with $Z$ the center of $G$. 

In \fullref{sec:hierarchies} we extend  \fullref{BNSintro} 
further to hierarchies, and we study the isomorphism type of $\ker\varphi$ through its action on trees.
We apply these results in \fullref{sec:mappingtori} to compute the
Bieri-Neumann-Strebel invariant for mapping tori of polynomially
growing free group automorphisms: with the notations of \fullref{mainintro},   a  nonzero $\varphi:G\to\R$ represents an element of $\Sigma(G)$ if and only if no $\varphi(t_i)$ equals 0, so \emph{$\Sigma(G)$ is the complement of a finite union of codimension 1 spheres.}

In \fullref{rank} we compute the rank of $\ker\varphi$ when it is
finitely generated, thus completing the proof of 
\fullref{mainintro}. 
Another approach to   this computation is to use the Alexander norm
(see the discussion in \cite[Section~1.6]{DowKapLei13a}).

In \fullref{cen} we consider mapping tori of finite order automorphisms of free groups $\F_n$, with $n\ge2$. These groups are exactly the non-solvable $\gbs_1$ groups having a non-trivial center \cite{Lev13gbsrank} (a $\gbs_1$ group  with trivial center has $\Sigma(G)$ empty). Given such a $G$, we determine for which values of $n$ and $k$ one can view $G$ as the mapping torus of an element of order $k$ in $\Out(\F_n)$, and also for which $n,k$ there is a subgroup $G_0\subset   G$  of index $k$, isomorphic to $\F_n\times\Z$,  and containing the center of $G$. In particular, we show that the set of ranks of   fibers is an arithmetic progression. 

\bigskip

The authors thank Ilya Kapovich for his interest and encouragement,
  Goulnara Arzhantseva for translating Khramtsov's example, and Ralph Strebel for useful comments.

 \section{The BNS invariant of a graph of groups}\label{sec:graphsofgroups}
 
 We first recall the definition and some properties of the BNS
 invariant $\Sigma(G)$. See \cite{BieNeuStr87} and \cite{Str12} for details.
 
Given a finitely generated group $G$, let  $S(G)$ be the sphere consisting of projective classes of non-zero homomorphisms $\varphi:G\to
\R$. We write $[\varphi]$ for the class of $\varphi$, so that $[\varphi]=[\psi]$ if and only if $\varphi=\lambda\psi$ with $\lambda>0$. The dimension of $S(G)$ is $b_1(G)-1$, with $b_1(G)$ the first Betti number.   If $X\subset \Sigma(G)$, let $-X$ be the set of classes $[-\varphi]$ for $[\varphi]\in X$.

A class $[\varphi]\in S(G)$ is \emph{discrete} if every homomorphism
in the class has discrete image, which is equivalent to saying that
$[\varphi]$ contains a surjection onto $\Z$.

For $H\subset G$, let $S(G,H)$ denote the subsphere $\{[\varphi]\in
S(G)\mid H\subset\ker\varphi\}$.    Its complement in $S(G)$ is denoted $S(G,H)^c$.
Let $\iota_H\from H\into G$ denote the inclusion map, and let
$\iota_H^*\from S(G,H)^c \to S(H)$ be the restriction map $ [\varphi]\mapsto [\varphi|_H]$.

 Let $\Cay(G,\mathcal{G})$ be the Cayley graph of $G$ with respect to a
finite generating set $\mathcal{G}$.   We identify its vertex set with $G$.
Given $\varphi\from G\to\R$, let $\Cay(G,\mathcal{G})^{[0,\infty)}$
denote the induced subgraph of $\Cay(G,\mathcal{G})$ containing the
vertices   in $\varphi^{-1}([0,\infty))$.
One of the several equivalent definitions of $\Sigma(G)$, see \cite{Str12}, is that $[\varphi]\in S(G)$ belongs to $\Sigma(G)$ if and
only if $\Cay(G,\mathcal{G})^{[0,\infty)}$ is connected for
some, equivalently, for every, finite generating set $\mathcal{G}$ of $G$. The set $\Sigma(G)$ is open in $S(G)$. 
See \cite{BieNeuStr87,
  Bro87, Lev94,Mei90, Str12} for alternate definitions.

\cite[Theorem~B1]{BieNeuStr87} states the following: \emph{if   $N\lhd
  G$, with $G/N$ abelian, then $N$ is finitely generated
  if and only if $S(G,N)\subset\Sigma(G)$}. 
Applied to  $N=[G,G]$, this shows that $\Sigma(G)=S(G)$ if and only if $[G,G]$ is finitely generated. 
  When $N=\ker\varphi$, with 
  $\varphi\from G\onto\Z$, then
$S(G,N)=\{[\varphi],-[\varphi]\}$, so $\ker\varphi$ is
finitely generated for discrete $[\varphi]$ if and
only if both $[\varphi]$ and $-[\varphi]$ are in $\Sigma(G)$.

As shown in 
\cite[Corollary~F]{BieNeuStr87},   results of
    Thurston \cite{Thu86}  imply that, 
  if $G$ is the fundamental group of a compact 3--manifold, then
  $\Sigma(G)=-\Sigma(G)$ and is a disjoint union of finitely many open convex rational polyhedra in $S(G)$.

 By a \emph{splitting} of a group $G$, we will mean a one-edge splitting, i.e.\ a decomposition of $G$ as an amalgam $G_1 *_A G_2$ or an HNN-extension $G_1*_A$. We denote by $\iota_{G_i}$ the natural inclusion from $G_i$ into $G$.
 
  The following lemma says that, if $\varphi:G\to\R$ does not vanish on $A$, and its restriction to each $G_i$ represents an element of $\Sigma(G_i)$, then $[\varphi]$ is in $\Sigma(G)$.

\begin{lemma}\label{amalginSigma}
\hfill
\begin{enumerate}

  \item  Let $G$ be an amalgamated free product $G_1 *_A G_2$, with $G_1,G_2$ 
    finitely generated. Then:
\[S(G,A)^c\cap (\iota_{G_1}^*)^{-1}(\Sigma(G_1))\cap(\iota_{G_2}^*)^{-1}(\Sigma(G_2))\subset\Sigma(G).\]\label{item:amalgam}

\item Let $G$ be an HNN-extension $\langle G_1,t \mid t^{-1}at =
  \sigma(a)\text{ for } a\in A\rangle$ with finitely generated base group $G_1$, 
subgroup $A<G_1$, and injection $\sigma\from  A\into G_1$. Then:
\[S(G,A)^c\cap (\iota_{G_1}^*)^{-1}(\Sigma(G_1))\subset \Sigma(G).\]\label{item:HNN}
\end{enumerate}
\end{lemma}

\begin{proof}
Case (\ref{item:amalgam}) is an immediate corollary of
\cite[Lemma~B1.14]{Str12}.
We prove case (\ref{item:HNN}), which is very similar.    Assume $[\varphi]\in S(G,A)^c\cap (\iota_{G_1}^*)^{-1}(\Sigma(G_1))$.

Take a finite generating set for $G_1$ and append $t$ to get a finite
generating set for $G$. Consider the corresponding Cayley graph.
Since $ [\varphi]\in (\iota_{G_1}^*)^{-1}(\Sigma(G_1))$, it follows trivially from the definitions that for every $g\in G$
  and $b_1,b_2$ in $G_1$   there exists a path from $gb_1$ to
  $gb_2$ in $gG_1$ such that every vertex $x$ along the path satisfies
  $\varphi(x)\geq\min(\varphi(gb_1),\varphi(gb_2))$.

An element $g\in G$ can be expressed $g=b_0t^{\varepsilon_1}b_1\cdots
t^{\varepsilon_m}b_m$ with $\varepsilon_i=\pm 1$ and $b_i\in G_1$. 
The minimal such $m$ is called the \emph{syllable length of $g$}.
We prove by induction on the syllable length that if $\varphi(g)\geq 0$ then there
exists a   $\varphi$--non-negative path  from 1 to $g$, i.e.\ a path  such that every vertex $x$ on the path
satisfies $\varphi(x)\geq 0$. This claim implies $[\varphi]\in\Sigma(G)$.

The claim is true for syllable length 0 by the hypothesis that
$[\varphi|_{G_1}]\in\Sigma(G_1)$.
Now suppose $g$ has syllable length   $m>0$ and $\varphi(g)\geq 0$, and
suppose that the claim is true for shorter syllable length.
Let $g=b_0t^{\varepsilon_1}b_1\cdots
t^{\varepsilon_{m}}b_{m}$
and $g'=b_0t^{\varepsilon_1}b_1\cdots
t^{\varepsilon_{m-1}}b_{m-1}$.
Since $[\varphi]\notin S(G,A)$ we can choose $a\in A$ such that $\varphi(a)\geq\max(-\varphi(g'),-\varphi(g't^{\varepsilon_m}))$.

If $\varepsilon_m=1$ then $g=g'aa^{-1}tb_m=g'at\sigma(a^{-1})b_m$, where
$\varphi(g'a)$ and $\varphi(g'at)$ are non-negative and $g'a$ has syllable length   $<m$.
By the induction hypothesis there is a $\varphi$--non-negative path
from 1 to $g'a$. 
Concatenating the $t$--edge gives a $\varphi$--non-negative path from
1 to $g'at$.
Finally, the   remark  of the second paragraph implies that we can concatenate
this path with a path in $g'atG_1$ from
$g'at$ to    $g=g'at\sigma(a^{-1})b_m$ to get a $\varphi$--non-negative
path from 1 to $g$. 

If $\varepsilon_m=-1$ the argument is similar using the relation
$g=g't^{-1}aa^{-1}b_m=g'\sigma(a)t^{-1}a^{-1}b_m$ and the fact that $\varphi(\sigma(a))=\varphi(a)$.   The path passes through $g'\sigma(a)$ and $g'\sigma(a)t^{-1}$.
\end{proof}

\begin{corollary} \label{dansS}
Let $G$ be the fundamental group of a finite graph of groups, with every vertex group finitely generated. Let $\varphi:G\to \R$ be a homomorphism. 
Assume:
\begin{enumerate}
\item
  $\varphi $ is non-trivial on every edge group, and 
\item  $[\varphi_{ | G_v}]$ belongs to $\Sigma(G_v)$ for every vertex group $G_v$.
\end{enumerate}
Then $[\varphi]\in \Sigma(G)$.
\end{corollary}

\begin{proof}
 By induction on the number of edges, writing $G$ as an amalgam or an HNN-extension of groups which are fundamental groups of graphs of groups with fewer edges.
\end{proof}

We now study the converse, so we assume $[\varphi]\in \Sigma(G)$ and we consider restrictions of $\varphi$ to edge and vertex groups. To deduce (1), we simply need some non-degeneracy assumptions.

\begin{definition}
A graph of groups $\Gamma$ is said to be \emph{reduced} if, given an edge $e$ with distinct endpoints $v_1,v_2$, the inclusions $G_e\into G_{v_i}$ are proper.
\end{definition}

  In particular, an HNN-extension is always reduced. An amalgam $G_1*_AG_2$ is reduced if and only if it is non-trivial ($A\ne G_1,G_2$).

If $\Gamma$ is not reduced, we may make it reduced by iteratively collapsing edges making it non-reduced. This does not introduce new  vertex groups since $G_{v_1}*_{G_{v_1}}G_{v_2}$ is isomorphic to $G_{v_2}$.

\begin{definition}
An HNN extension $G=\langle G_1,t\mid t^{-1}at=\sigma(a)\text{ for
}a\in A\rangle$ is called \emph{ascending} if $A=G_1$ and \emph{descending}
if $\sigma(A)=G_1$.
It is \emph{strictly ascending} if it is ascending and not descending.
\end{definition}

Replacing the stable letter $t$ by its inverse reverses
ascending/descending.   We say that a graph of groups $\Gamma$ is not an ascending HNN-extension if it is not an HNN extension (it has more than one edge or more than  one vertex), or it is an HNN-extension but both $A$ and $\sigma(A)$ are proper subgroups of $G_1$.

\begin{proposition}
\label{amalgnotinSigma} 
Let $G$ be the fundamental group of a finite reduced graph of groups $\Gamma$, with $G$ finitely generated. Assume that $\Gamma$ is not an ascending HNN-extension. If $[\varphi]\in \Sigma(G)$, then   $\varphi $ is non-trivial on every edge group.
\end{proposition}

\begin{proof}
 The special case when $\Gamma$ has a single edge is Proposition~C2.13 of {\cite{Str12}}. Given any edge $e$ of $\Gamma$, collapse all other edges. The resulting graph of groups is not an ascending HNN-extension because $\Gamma$ is reduced, so we may   apply the special case to conclude that $\varphi$ is non-trivial on $G_e$.
\end{proof}

\begin{rem*}The proposition is wrong in the case of an ascending
  HNN-extension (see \fullref{sasc}).
Rather than assuming that $\Gamma$ is reduced and not an ascending HNN-extension, one could assume that its Bass-Serre tree is minimal and irreducible. 
\end{rem*}

The following example
shows that some more serious restriction is necessary in order to deduce the statement $[\varphi_{ | G_v}]\in\Sigma(G_v)$ (item (2) in \fullref{dansS}) from the assumption $[\varphi]\in\Sigma(G)$.

\begin{example} View  $G=\F_2\times\F_2$ as the amalgam of $\F_2\times\Z=\langle a,b\rangle \times \langle c\rangle$ with $\F_2\times\Z=\langle a,b\rangle \times \langle d\rangle$
over $\F_2=\langle a,b \rangle$. If $\varphi:G\to\R$ kills $d$ but none of $a,b,c$, then $[\varphi]\in\Sigma(G)$ but $[\varphi_{ | G_v}]\notin\Sigma(G_v)$ for $G_v=\langle a,b\rangle \times \langle d\rangle$ (see Theorem 7.4 of   \cite{BieNeuStr87}
 or Proposition A2.7 of  \cite{Str12}  for $\Sigma$ of a direct
 product). 
This may be generalized to groups of the form $A\times \F_2$ with $\Sigma(A)\ne S(A)$.
\end{example}

\begin{proposition}\label{inSigmavert} 
Let $G$ be the fundamental group of a finite 
graph of groups $\Gamma$, with all edge and vertex groups  finitely generated. 
Assume that $[\varphi]\in \Sigma(G)$,  and $[\varphi_{ | G_e}]\in\Sigma(G_e)$ for every edge group $G_e$ (in particular,  $\varphi{( G_e)}\ne0$). Then 
 $[\varphi_{ | G_v}]\in\Sigma(G_v)$ for every vertex group $G_v$.
\end{proposition}

\begin{proof}
As in the proof of \fullref{dansS}, it suffices to prove the result when $\Gamma$ has a single edge:  $G=G_1 *_A G_2$ or  $G=\langle G_1,t \mid t^{-1}at =
  \sigma(a)\text{ for } a\in A\rangle$.
Since $\varphi$ does not vanish on $A$,  
it does not vanish on $G_i$, so $[\varphi|_{G_i}]\in S(G_i)$. 

First suppose $G=G_1*_A G_2$. 
Fix a finite generating set $\mathcal{A}$ of $A$, and choose finite generating sets
$\mathcal{G}_i$ for $G_i$   extending the respective images of $\mathcal{A}$, with $\mathcal{G}_i\setminus\mathcal{A} \subset  G_i\setminus A$.
Their union $\mathcal{G}=\mathcal{G}_1\cup\mathcal{G}_2$ is a finite generating set for $G$.

Take two vertices $g$ and $h$ in 
$\Cay(G_1,\mathcal{G}_1)^{[0,\infty)}$. We have to join them by a path contained in $\Cay(G_1,\mathcal{G}_1)^{[0,\infty)}$.

There exists a path $p$ joining them in
$\Cay(G,\mathcal{G})^{[0,\infty)}$, since $[\varphi]\in\Sigma(G)$.  
Write labels along $p$ as $g_1g_2\dots g_k$ with $g_i\in\mathcal{G}$.
If every $g_i$ is in $\mathcal{G}_1$, then $p$ is contained in $\Cay(G_1,\mathcal{G}_1)$ and we are done. Otherwise, by standard facts about free products with amalgamation, there exists a subpath  $g_{i+1}\dots g_j$ containing an edge with label in  $\mathcal{G}_2\setminus\mathcal{A}$ and representing an element of $A$.

Since $\mathcal{G}_1$ is an extension of $\mathcal{A}$, and
$\Cay(A,\mathcal{A})^{[0,\infty)}$ is connected because $[\varphi|_{A}]\in
 \Sigma(A)$, we may replace the
subpath $g_{i+1}\dots g_j$ of $p$ by a subpath contained in
$\Cay(A,\mathcal{A})^{[0,\infty)}\subset \Cay(G_1,\mathcal{G}_1)^{[0,\infty)}$.
This reduces the number of labels of $p$ not in $\mathcal{G}_1$.

Repeat this process until all labels of $p$ are in $\mathcal{G}_1$.
This gives a 
  path from $g$ to $h$ in
$\Cay(G_1,\mathcal{G}_1)^{[0,\infty)}$, so $[\varphi|_{G_1}]\in\Sigma(G_1)$. 
The same argument applies to $G_2$.

In the HNN extension case, choose a finite generating set $\mathcal{G}_1$
of $G_1$ extending $\mathcal{A}\cup\sigma(\mathcal{A})$.
Let $\mathcal{G}=\mathcal{G}_1\cup\{t\}$, and apply the same argument as
in the amalgamated product case (using Britton's lemma) to conclude $[\varphi|_{G_1}]\in\Sigma(G_1)$.
\end{proof}

Propositions \ref{amalgnotinSigma}  and \ref{inSigmavert}  do not provide a complete converse of \fullref{dansS}, because   \fullref{amalgnotinSigma} only guarantees    
$\varphi(G_e)\ne0$,
not $[\varphi_{ | G_e}]\in\Sigma(G_e)$. This motivates the following definition.

\begin{definition}\label{def:varsigma}
  Let $\varsigma$ be the class of finitely generated groups $G$ with
  $\Sigma(G)=S(G)$.
\end{definition}
Equivalently, $\varsigma$ is the class of finitely generated groups
with finitely generated commutator subgroup    \cite{BieNeuStr87}.

Recall that a group is \emph{slender} if every subgroup is finitely
generated.
Slender groups belong to $\varsigma$.
Examples of slender groups include finitely generated virtually
abelian groups and virtually polycyclic groups.

\fullref{dansS} and Propositions \ref{amalgnotinSigma} and \ref{inSigmavert} now imply:
\begin{corollary} \label{varsigmaamalgam}
Let $G$ be the fundamental group of a finite reduced graph of groups $\Gamma$, with $G$ finitely generated and edge groups in $\varsigma$.
 Assume that $\Gamma$ is not an ascending HNN-extension. Then  $[\varphi]\in \Sigma(G)$ if and only if    $\varphi $ is non-trivial on every edge group and  $[\varphi_{ | G_v}]\in\Sigma(G_v)$ for every vertex group $G_v$:
\[\pushQED{\qed} \Sigma(G)=\bigcap_{e\in\mathcal{E}\Gamma}S(G,G_e)^c\cap\bigcap_{v\in\mathcal{V}\Gamma}(\iota_{G_v}^*)^{-1}(\Sigma(G_v)).\hfill\qedhere\popQED\]
\end{corollary}

Note that all vertex groups are finitely generated, so $\Sigma(G_v)$ is defined. 

\begin{corollary} \label{graphofvarsigmagroups}
Let $G$ be the fundamental group of a finite reduced graph of groups $\Gamma$, with every vertex group in $\varsigma$. 
 Assume that $\Gamma$ is not an ascending HNN-extension. Then  $[\varphi]\in \Sigma(G)$ if and only if    $\varphi $ is non-trivial on every edge group. 
In particular, 
\[\Sigma(G)=-\Sigma(G)=\bigcap_{e\in\mathcal{E}\Gamma}S(G,G_e)^c\] is
the complement of a finite number of rationally defined subspheres. 
\qed
\end{corollary}

This applies in particular to graphs of virtually polycyclic groups.

\begin{remark}\label{sasc}
If  $G=\langle G_1,t \mid t^{-1}at =
  \sigma(a)\text{ for } a\in G_1\rangle$ is an \emph{ascending}
  HNN-extension then $S(G,G_e)$ consists of two points, the projective classes containing $\varphi_\pm\from
  G\to\Z$ defined by  $\varphi(G_1)=0$ and $\varphi(t)=\pm1$. 
When  $G_1$ is finitely generated, $\Sigma(G)$ contains $[\varphi_+]$, and  contains 
  $[\varphi_-]$ if and only if the extension is not strictly ascending    (Proposition 4.4 of \cite{BieNeuStr87}).   In particular,  \emph{$\Sigma(G)\neq-\Sigma(G)$ if $G$ is a strictly ascending HNN-extension with finitely generated base group.}
  \end{remark}
  We thus have:

\begin{corollary} \label{asc}
  If $\Sigma(G)\ne-\Sigma(G)$, in particular 
if $G$ is a strictly ascending HNN-extension with finitely generated base group, 
then every decomposition of $G$ as the fundamental group of a   finite reduced
graph of groups with vertex  
  groups in $\varsigma$ is a strictly ascending HNN-extension. \qed
\end{corollary}

\section{Generalized Baumslag-Solitar Groups}\label{sec:gbs}
Let $G$ be as in 
\fullref{graphofvarsigmagroups}. Assume furthermore that all
the inclusions   $G_e\into G_{v }$ map $G_e$ onto a finite index subgroup  (this is equivalent to the Bass-Serre tree of $\Gamma$ being locally finite). Then all edge and vertex groups are commensurable   (the  intersection of any two has finite index in both), so $\Sigma(G)= S(G,H)^c $ with $H$ any vertex or edge group. Moreover, a homomorphism $\varphi:G\to \R$ represents an element of $S(G,H)$ if and only if it factors through the topological fundamental group of $\Gamma$. We thus get:

\begin{corollary} Let $G$ be as in 
\fullref{graphofvarsigmagroups}, with all edge and vertex groups
commensurable. 
If $H$ is any vertex or edge group, then $\Sigma(G)= S(G,H)^c $ is the complement of a rationally defined subsphere whose dimension is $b-1$, with $b$ the first Betti number of the graph $\Gamma$. \qed
\end{corollary}

In particular, a rank $n$ \emph{generalized Baumslag-Solitar group}, or $\gbs_n$ group, is the
fundamental group $G$ of a finite graph of groups $\Gamma$ all of whose edge and vertex groups are isomorphic to 
$\Z^n$. When $n=1$ we simply say that $G$ is  a $\gbs$ group.

If $\Gamma$ is   an ascending HNN-extension, $G$ is solvable, so we obtain:

\begin{corollary} If $G$ is a non-solvable $\gbs_n$ group, then $\Sigma(G)= S(G,H)^c $ with $H$ any edge or vertex group. \qed
\end{corollary} 

In certain cases, we may relate $\Sigma(G)$ to the center $Z(G)$.

 \begin{proposition}\label{prop:gbs1center}
Let $G$ be a non-solvable $\gbs$ group. Then:
\[\Sigma(G)=S(G,\gpcenter(G))^c.\]
 If the center is trivial, $\Sigma(G)=\emptyset$. If not, $\Sigma(G)$
 is the complement of a codimension 1 rationally defined subsphere.
 \end{proposition}
 
The solvable $\gbs$ groups are 
   $\Z$ and the Baumslag-Solitar groups $BS(1,k)=\langle a,t\mid t^{-1}at=a^k\rangle$.
 
 \begin{proof} This follows from 
   \cite[Propositions 2.5 and 3.3]{Lev07}. 
If the center is trivial,
   edge groups have finite image in
   the abelianization of $G$ so $\Sigma(G)=\emptyset$. 
If there is a center, it is infinite cyclic, contained in every edge group, and maps injectively into the abelianization.
 \end{proof}

Going back to arbitrary $n$, suppose that $G=\Z^n\times \F_k$ for some $k\ge2$. Then all edge and vertex groups are equal to the $\Z^n$ factor, which is the center, and $\Sigma(G)= S(G,Z(G))^c $. This may be generalized to $\gbs_n$ groups which are virtually $\Z^n\times \F_k$ (such groups are called unimodular, see below).

 \begin{proposition}\label{prop:nonsolvableunimodularcenter}
If the   $\gbs_n$ group $G$ is non-solvable and
unimodular   (i.e.\  $G$ is  virtually $\Z^n\times \F_k$ with
   $k\ge2$), 
then $\Sigma(G)=S(G,\gpcenter(G))^c$.
 \end{proposition}
 \begin{proof}  Consider the action of $G$ on the Bass-Serre tree $T$ of $\Gamma$ (we may assume that this action is minimal: there is no proper invariant subtree). 
 By a standard argument, the center of $G$  acts as  the identity on   $T $ (it is elliptic because 
 $T$ is not a line, and its fixed subtree is $G$-invariant), so a $\varphi$ that does not vanish on the center is in $\Sigma(G)$. 
We have proved $S(G,\gpcenter(G))^c\subset \Sigma(G)$ (this also follows from Proposition A2.4 of \cite{Str12}).

 Let $G_0=\Z^n\times \F_k$ have finite index in $G$. It acts minimally on $T$, 
 and as above its center $Z(G_0)=\Z^n$ acts as the identity on $T$.
 We deduce that the set of elements of $G$ acting as the identity on  $T $ is a normal subgroup $N$ isomorphic to $\Z^n$; it has finite index in every edge stabilizer, and contains $Z(G_0)$ with finite index.

If   $[\varphi]\in\Sigma(G)$, then $\varphi$ does not vanish on edge stabilizers, so there is an $x\in  Z(G_0)$  with $\varphi(x)\ne0$. This $x$ has finitely many conjugates in $G$, all contained in the abelian group $N$, and the product of these conjugates is a central element of $G$ which is not killed by $\varphi$. This proves $\Sigma(G)\subset S(G,\gpcenter(G))^c$.
 \end{proof}

We may interpret these results in terms of the \emph{modular representation $\Delta$} of $G$. 
Let $G$ be a non-solvable $\gbs_n$ group, and let $G_e\cong\Z^n$ be an edge
group. 
If $g\in G$, conjugation by $g$ induces an isomorphism $$g^{-1}G_eg\cap G_e\to G_e\cap gG_eg^{-1}$$
between finite index subgroups of $G_e$. We define $\Delta(g)$ as the
class of this isomorphism in the abstract commensurator  of $G_e$ (isomorphic 
  to $\GL_n(\Q)$).
  The groups of \fullref{prop:nonsolvableunimodularcenter}
are those for which the image  $\Delta(G)$ is a finite subgroup of $\GL_n(\Q)$.

Since $G_e\cong\Z^n$, we may
  view $\Delta$ as an action of $G$ on the vector space  $G_e\otimes \Q$.
The codimension of $S(G,G_e)$ is the
dimension of the space of invariant vectors for the action on the dual
space. The rank of the center of $G$  is the dimension of the space of
invariant vectors for the action on the   space itself. Unfortunately,
these dimensions may be different when $n>1$, so the analogue of
\fullref{prop:gbs1center} can not be true in general.

\section{Hierarchies}\label{sec:hierarchies}

The goal of this section is to study $\Sigma(G)$, and $\ker \varphi$ for $\varphi:G\to\R$, when $G$ admits a suitable hierarchical decomposition.

\begin{definition}
  A $\mathcal{P}$--\emph{hierarchy} for a group $G$ is an iterated
  splitting over subgroups in the class $\mathcal{P}$. 
  \end{definition}

    More precisely,  a hierarchy may be viewed as a finite rooted tree. Each vertex $\alpha$ carries a group $G(\alpha)$, with the root carrying $G$. If $\alpha$ is not a leaf, it carries 
  a non-trivial    
splitting of  $G(\alpha)$ as an amalgamated product or
  HNN extension over a group in the class $\mathcal{P}$, called an edge group of the hierarchy. Moreover, $\alpha$ has one descendant $\beta$
for each factor group in the splitting (so $\alpha$ has one or two descendants), and $G(\beta)$ is the corresponding factor. 

The groups carried by the leaves are called the leaf groups of the
hierarchy. 
The number of splittings in the hierarchy is equal to the number of non-leaf vertices.

A  $\mathcal{P}$--hierarchy is a \emph{good $\mathcal{P}$--hierarchy}
if all groups $G(\alpha)$ are finitely generated, and
no splitting in the hierarchy is an ascending HNN-extension.
Recall that $\varsigma$ is the class of finitely generated groups $H$ with
  $\Sigma(H)=S(H)$.

\begin{theorem}\label{thm:varsigmahierarchy}
  Suppose $G$ admits a good $\varsigma$--hierarchy,   
   with splittings 
   over   subgroups $\{A_i\mid
i\in\mathcal{I}\}$, and   leaf groups 
$\{H_j\mid
j\in\mathcal{J}\}$.

Then:
\[\Sigma(G)=\bigcap_{i\in\mathcal{I}} S(G,A_i)^c \cap
\bigcap_{j\in\mathcal{J}}(\iota_{H_j}^*)^{-1}(\Sigma(H_j)).\]
\end{theorem}
\begin{proof}
  The proof is by induction on the number of splittings in the
  hierarchy.
If there are no splittings then $G$ itself is a leaf   group and the result is clear.
Otherwise, consider the first splitting in the hierarchy, which is of
the form $G=G_1*_AG_2$ or $G=G_1*_A$.
Since $A\in\varsigma$ and the splitting is not an ascending HNN-extension, by  \fullref{varsigmaamalgam},   
\[\Sigma(G)=S(G,A)^c\cap
\bigcap_{k\in\mathcal{K}}(\iota^*_{G_k})^{-1}(\Sigma(G_k)),\]
where $\mathcal{K}=\{1,2\}$ in the amalgamated product case and
$\mathcal{K}=\{1\}$ in the HNN case.
For each $k$,   consider the   vertex carrying $G_k$ in the hierarchy for $G$. Its descendant subtree 
is a hierarchy for $G_k$ satisfying the hypotheses of the
theorem and having strictly fewer splittings.
The theorem then follows by the induction hypothesis.
\end{proof}

\begin{corollary} \label {pasnul}
 If $G$ admits a good $\varsigma$--hierarchy with leaf groups in $\varsigma$,  then   $\Sigma(G)=\bigcap_{i\in\mathcal{I}} S(G,A_i)^c$ is
 the complement of a finite number of rationally defined subspheres.
 
In particular, $\Sigma(G)=-\Sigma(G)$, so $G$ does not admit a
decomposition as a strictly ascending HNN extension with finitely generated base group (see \fullref{sasc}).  \qed
\end{corollary}

 \begin{theorem}\label{thm:Zhierarchy}
   Suppose $G$ admits a good $\Z$--hierarchy   with 
leaf groups    $\{H_j\mid j\in\mathcal{J}\}$ belonging to 
   $\varsigma$. Then:
\begin{enumerate}
  \item $\Sigma(G)^c$ is a finite union of rationally defined subspheres.\label{item:complementofsubspheres}
  \item If $[\varphi]\in\Sigma(G)$, then $\ker \varphi$ is
  a free product of groups,   each isomorphic to   some  $\ker\varphi_{|H_j}$ with
  $\varphi(H_j)\ne0$.
If $[\varphi]$ is discrete this is a finite free product.\label{item:fgkernel}
\item If $[\varphi]\in \Sigma(G)^c$ is discrete, and  every group in the hierarchy has first Betti number at least 2,
 then $\ker \varphi$ surjects onto
  $\F_\infty$ (an infinitely generated free group).\label{item:surjectiontofinfinity}
  \end{enumerate}
 \end{theorem}

  The first assertion follows from
\fullref{thm:varsigmahierarchy}.
Before proving the others, 
we need a simple  orbit counting lemma:
\begin{lemma}\label{lem:orbitcounting}
 Let $G$ act transitively on a set $X$. 
Let $K$ be a normal subgroup, with $p\from G\to G/K$ the quotient map. 
Let $x\in X$ be arbitrary, with stabilizer $G_x$. 
The number of $K$--orbits in $X$ equals the index of $KG_x$ in $G$, and also the index of $p(G_x)$ in $G/K$. 
\end{lemma} 
 \begin{proof}
$G$ acts transitively on the set of
   $K$--orbits, with the stabilizer of $Kx$ equal to $KG_x$ (which is
   a subgroup because $K$ is normal). 
The map $p$ induces a bijection between   cosets mod $KG_x$ in $G$ and cosets mod $p(G_x)$ in $G/K$.
 \end{proof}

 \begin{proof}[Proof of \fullref{thm:Zhierarchy}]
   Assertions (2) and (3) 
are   proved by induction on the number of splittings in the
hierarchy. 
If there are no splittings then $\Sigma(G)=S(G)$ and the result is clear.
Otherwise the first splitting in the hierarchy for $G$ corresponds to
either an amalgamated product $G=G_1*_A G_2$ 
or an HNN-extension $G=G_1*_A$. The $G_i$'s  have shorter
hierarchies satisfying the hypothesis of the theorem. 

  Let $K=\ker\varphi$, so $G/K\simeq \varphi(G)$ is a subgroup of $\R$.
We study $K$ through its action on the Bass-Serre tree $T$ of the first  splitting in the hierarchy. Since $K$ is normal, vertex stabilizers are conjugate to $K\cap G_i=\ker
\varphi_{|G_i}$ for $i=1$ or 2; edge stabilizers are conjugate to $K\cap A=\ker
\varphi _{|A}$.

Consider the image in $G/K$ of the stabilizer $G_v$ of a point $v\in
T$ for the action of $G$.
It is trivial if $G_v<K$ (i.e.\    if $\varphi(G_v)=0$) and infinite otherwise.
Applying \fullref{lem:orbitcounting} with $X=Gv$, this implies that
$Gv$ splits into infinitely many $K$--orbits if $\varphi(G_v)=0$.
 The converse is also true if $[\varphi]$ is discrete:    if $\varphi(G_v)\ne0$, the number of $K$-orbits in $Gv$ is the index of $\varphi(G_v)$ in $\varphi(G)$, which is finite since $\varphi(G)\simeq\Z$.
In particular, $T/K$ has infinitely many edges if $\varphi(A)=0$, finitely many if $[\varphi]$ is discrete and $\varphi(A)\ne0$.

First suppose  $\varphi(A)\neq 0$, so $\varphi(G_i)\ne0$. In this case edge stabilizers for the $K$--action are
trivial ($\varphi$ is injective on $A$ since $A\cong\Z$).
Thus, $K$ is a free product of groups isomorphic to $\ker
\varphi_{|G_i}$ (and both factors occur in the case of an amalgam).
Moreover, if $[\varphi]$ is discrete then $T/K$ is finite, so this is
a finite free product.

If $[\varphi]\in \Sigma(G)$, its restrictions to $G_i$ belong to $\Sigma(G_i)$ by 
\fullref{varsigmaamalgam}, and the 
 theorem follows by induction.
  If $[\varphi]\in \Sigma(G)^c$, we have $[\varphi_{ | G_i}]\in \Sigma(G_i)^c$ for some $i$. If $[\varphi]$ is discrete and Betti numbers are $\ge2$,  $\ker
\varphi_{|G_i}$   maps onto $\F_\infty$ by induction, and so does $\ker\varphi$.

Now suppose $\varphi(A)=0$. Then $[\varphi]\in \Sigma(G)^c$ by 
\fullref{varsigmaamalgam}, so the second assertion of the theorem is proved. 

To complete the proof, we construct a surjection from $K$ to $\F_\infty$, assuming 
  that $[\varphi]$ is discrete, $\varphi(A)=0$, and Betti numbers are $\ge2$.
  
As explained above, 
$T/K$  has infinitely many edges. If $v$ is a vertex of $ T$ with $\varphi(G_v)=0$, the valence of $v$ in $T/G$ is the same as in $T/K$ because $G_v$ is conjugate to   some $G_i$, and  the index of $K\cap A$ in $K\cap G_i$ equals the index of $A$ in $G_i$   since $G_i\subset K$.

We distinguish between the amalgamated product and HNN cases.

In the amalgamated product case, if $\varphi(G_1)$ and $\varphi(G_2)$ are both nonzero there are
finitely many vertices  in $T/K$   (but infinitely many edges). Viewing $K$ as the fundamental group of the graph of groups $T/K$,  we get a surjection from $K$ onto
$\F_\infty$ by killing all the vertex groups.

If, say, $\varphi(G_1)=0$ (and therefore $\varphi(G_2)=\varphi(G)\ne0$),
then $T/K$ is a star with exactly one vertex of type 2  (carrying a conjugate of $G_2$) and
infinitely many valence 1 vertices of type 1. 
The groups carried by these terminal vertices are conjugates of
$G_1$.
By hypothesis, $G_1$ has first Betti number at least 2, so there exists a non-zero map from each terminal vertex
group to $\Z$ killing the incident edge group (hence also  its normal closure).
We map $K$ onto $\F_\infty$ by killing the central type 2 vertex
group, hence all edge groups, and piecing together the surjections to
$\Z$ of the resulting quotients of the infinitely many terminal vertices.

In the HNN case, $G\cong\langle G_1,t\mid  {t}^{-1}At=\sigma(A)\rangle$.
If $\varphi(G_1)$ is non-zero then,   as before,  there are finitely many vertices and 
 we get a surjection from $K$ onto
$\F_\infty$ by killing all the vertex groups.

If $\varphi(G_1)=0$, the image of $\varphi$ is generated by
$\varphi(t)$.
The valence of vertices in $T/K$ is the same as that in $T/G$, which
is 2, so $T/K$ is a line. Denote its edges $e(i)$, for $i\in\Z$.
The vertex stabilizers are conjugates of $G_1$.
To map $K$ onto $\F_\infty$, one should not kill all edge groups $G_{e(i)}$ because
vertex groups might be killed too, but one can kill every other edge group $G_{e(2i)}$.
This maps $K$ onto an infinite free product, and we check that each factor may be mapped onto $\Z$.

Such a factor has the form
$H=B*_{\langle c\rangle}D$, where each group $B$,  $D$ is isomorphic to the  quotient of $G_1$ by an incident edge group ($G_{e(2i)}$ or $G_{e(2i+2)}$).
Again, since the first Betti number of $G_1$ is at least 2, there exist
epimorphisms $\rho\from B\to\Z$ and $\tau\from D\to\Z$.
If $\rho(c)=\tau(c)=0$, then $H$ 
maps onto $\F_2$.
Otherwise, define a non-zero map from $H$ to $\Z$ as 
$\tau(c)$ times $\rho$ on $B$ and $\rho(c)$ times $\tau$ on $D$. 
 \end{proof}

\section{Mapping Tori of Polynomially Growing Free Group Automorphisms}\label{sec:mappingtori}

Given $\alpha\in\Aut(\F_n)$, we denote by $G_\alpha$ the mapping torus 
$$G_\alpha=\F_n\rtimes_\alpha\Z=\langle \F_n,t \mid t^{-1}gt=\alpha(g)\rangle.$$
Up to isomorphism, it only depends on the outer automorphism $\Phi\in\Out(\F_n)$ represented by $\alpha$, so we often write $G_\Phi$ rather than $G_\alpha$.

Let   $\varphi:G_\alpha\to\Z$ be the map sending $\F_n$ to $0$ and $t$ to 1.
We say that 
$\varphi$ (and any multiple of it) is a \emph{fibration} with \emph{fiber} $\F_n$ and \emph{monodromy} $\alpha$ (or $\Phi$).

Recall that $\alpha$ (or $\Phi$) is polynomially growing if, for any $g\in\F_n$,  the length of $\alpha^k(g)$ grows polynomially.

\begin{remark}
$\alpha$ has a well-defined degree of polynomial growth $d(\alpha)$, which is the maximal degree of growth of   
the length, with respect to some fixed 
word metric on $\F_n$, of the shortest conjugate of 
$\alpha^k(g)$.   Macura \cite{Mac02} has shown that $d(\alpha)$ is a quasi-isometry
  invariant of $G_\alpha$. In particular, if $\alpha\in\Aut(\F_n)$ and $\beta\in\Aut(\F_m)$ yield isomorphic mapping tori $G_\alpha\simeq G_\beta$, then $d(\alpha)=d(\beta)$.
\end{remark}

The main result of this section is the following:

 \begin{theorem}\label{main}
Let $G=G_\alpha=\F_n\rtimes_\alpha\Z$ for $\alpha\in\Aut(\F_n)$ polynomially growing, with $n\ge2$. 
There exist elements $t_1,\dots, t_{n-1}$ in $G\setminus \F_n $ (not
necessarily distinct), such that:
\[\Sigma(G)=-\Sigma(G)=\bigcap_{i}S(G,t_i)^c.\]

More precisely, for each non-zero $\varphi\from G\to \R$:
 \begin{itemize}
 \item If some $\varphi(t_i)$ is $0$, then
   $[\varphi]\notin\Sigma(G)$. If $[\varphi]$ is discrete then $\ker\varphi$ virtually
   surjects onto $\F_\infty$.
  \item If no $\varphi(t_i)$ is $0$, then 
    $[\varphi]\in\Sigma(G)\cap -\Sigma(G)$ and $\ker\varphi$ is free. 
If $[\varphi]$ is discrete then $\ker \varphi$ has finite rank. 
 \end{itemize}
\end{theorem}

\begin{corollary}\label{BNS}
If  the first Betti number of $G$ is at least 2,
then  $S(G)\setminus\Sigma(G)=\cup_{i}S(G,t_i)$ is a
non-empty collection of rationally defined great subspheres; in particular, $\Sigma(G)\ne S(G)$. \qed
\end{corollary}

Using \fullref{sasc}, we also have:

\begin{corollary}\label{corollary:noascendingHNN}
  A mapping torus of a polynomially growing free group automorphism
  does not admit a decomposition as a strictly ascending HNN extension with finitely generated base group.
 \qed
\end{corollary}

These results do not hold for exponentially growing automorphisms, as
evidenced by the group $G$ constructed by Leary, Niblo, and Wise \cite{LeaNibWis99}. It  is the mapping torus of an automorphism of $\F_3$, and also of an injective, non surjective, endomorphism of $\F_2$. It does not satisfy $\Sigma(G)=-\Sigma(G)$, and there exists a discrete $[\varphi]$   such that  $\ker\varphi$ is a strictly increasing union of 2-generated subgroups, hence is infinitely generated but does not virtually surject onto $\F_\infty$.

\subsection{UPG automorphisms}

Recall the natural map $\tau:\Out(\F_n)\to\GL_n(\Z)$ recording the action of automorphisms on the abelianization of $\F_n$.

\begin{definition}[\cite{BesFeiHan00}]   
  $\upg(\F_n)$ is the set of polynomially growing elements  
  $\Phi\in\Out(\F_n)$ that have unipotent image in $\GL_n(\Z)$, and $\widehat\upg(\F_n)$ is the preimage of $\upg(\F_n)$ in $\Aut(\F_n)$.
\end{definition}

\begin{remark}\label{ki} 
If $n\ge1$ and $\Phi$ is UPG, then 1 is an  eigenvalue 
of $\tau(\Phi)$, and this guarantees  
  that $G_ {\Phi}$ has
  first Betti number at least $2$.
In particular, given any $g\in G_ \Phi$, there is a nontrivial map
from $G_\Phi$ to $\Z$ killing $g$. 
\end{remark}

\begin{lemma}[{\cite[Corollary~5.7.6]{BesFeiHan00}}]\label{lem:upg}
 Every polynomially growing element of $\Out(\F_n)$ has a power in $\upg(\F_n)$. \qed
\end{lemma}

\begin{lemma}
     If  $\Phi\in\upg(\F_n)$ has finite order, it is the identity.
\end{lemma}

\begin{proof} Being unipotent and of finite order, $\tau(\Phi)$ is trivial, so $\Phi\in\ker\tau$. But $\ker\tau$ is torsion-free \cite{BauTay68}. 
\end{proof}

The other relevant fact about unipotent polynomially growing automorphisms that we
need is the existence of invariant free splittings of $\F_n$.
  \begin{proposition}\label{prop:poly}
Let $n\ge2$. 
Every $\Phi\in\upg(\F_n)$ has a representative $\alpha\in\Aut(\F_n)$ such that one of the following holds: 
   \begin{enumerate}
 \item There exists a non-trivial $\alpha$-invariant splitting $\F_n=B_1*B_2$, so    $\alpha=\alpha_1*\alpha_2$ with $\alpha_i=\alpha_{ | B_i}\in\Aut(B_i)$. 
  \label{item:separating}
 \item There exists a non-trivial splitting $\F_n=B_1*\langle x\rangle$, where $B_1\simeq \F_{n-1}$ is  $\alpha$--invariant, 
 and $\alpha(x)=xu$ with $u\in B_1$. We denote $\alpha_1=\alpha_{ | B_1}$.
 \label{item:nonseparating}
 \end{enumerate} 
\end{proposition}
\begin{proof}
This is a consequence of Bestvina, Feighn, and Handel's train track
theory for free group automorphisms. 
By \cite[Theorem~5.1.8]{BesFeiHan00}, there exists a graph $\Gamma$
with valence at least 2 and fundamental group $\F_n$, a homotopy
equivalence $f\from \Gamma\to\Gamma$ inducing $\Phi$ on the
fundamental group, and a filtration   $\emptyset=\Gamma_0\subset
\Gamma_1\subset\dots\subset \Gamma_k=\Gamma$ satisfying several
properties.
Relevant for us are that $f$ fixes the vertices of $\Gamma$,
for every $i$ the stratum $\Gamma_i\setminus\Gamma_{i-1}$ is a single edge $E_i$,
and $f(E_i)=E_iu_i$ where $u_i$ is a loop in $\Gamma_{i-1}$.

The proposition is proven by considering the topmost stratum $E_k$ of the
filtration.
If $E_k$ is a separating edge then we are in case
(\ref{item:separating}), and $B_1$ and $B_2$ are the fundamental
groups of the two components of $\Gamma_{k-1}$.
If $E_k$ is a non-separating edge then we are in case
(\ref{item:nonseparating}), and $B_1$ is the fundamental group of
$\Gamma_{k-1}$.
\end{proof}

\begin{lemma}\label{lem:upghierarchy}
  If $\Phi\in\upg(\F_n)$ then $G_\Phi$ admits a good $\Z$--hierarchy
  with $\Z^2$ leaves and
   $n-1$ splittings.
Moreover, each edge group  in the hierarchy has trivial intersection with the fiber $\F_n$, and  each vertex group has first Betti number at least 2.
\end{lemma}
\begin{proof}
The lemma is proved by induction on $n$.
If $n=1$ then $G_\Phi=\Z^2$ and we take the trivial hierarchy, which
has $n-1=0$ splittings.
Otherwise, we construct the first splitting in the hierarchy as follows. Apply \fullref{prop:poly}. In the first case, we   write  $G_\alpha=G_{\alpha_1}*_{\langle t\rangle}G_{\alpha_2}$. In the second case, we have $$G_\alpha=\langle B_1,x,t\mid t^{-1}bt=\alpha_1(b),\, t^{-1}xt=xu \rangle=\langle G_{\alpha_1},x\mid x^{-1}tx=tu^{-1}\rangle,$$
and   we consider the HNN-extension
$G_\alpha=G_{\alpha_1}*_{\Z}$   
(which is not ascending since
$G_{\alpha_1}$ is the mapping torus of an automorphism of $B_1\cong\F_{n-1}$, so
it is not cyclic).

By the induction hypothesis, each $G_{\alpha_i}$ admits a good 
$\Z$--hierarchy   with $\Z^2$ leaves.
Take the hierarchy for $G$ consisting of the   splitting we just constructed, and then the
hierarchies for the $G_{\alpha_i}$.
In the amalgamated product case the number of splittings is 
$1+(\mathrm{rank}(B_1)-1)+(\mathrm{rank}(B_2)-1)=n-1$.
In the HNN extension case the number of splittings is
$1+(\mathrm{rank}(B_1)-1)=n-1$.

The claim about edge groups is clear from the way the hierarchy is constructed, and the Betti number claim follows from \fullref{ki}: since $B_i$ as in \fullref{prop:poly} is   an $\alpha$-invariant  free factor of $\F_n$, the restriction $\alpha_i$ is UPG.
\end{proof}

\subsection{Proof of \fullref{main}}
We first note that it  suffices to prove the theorem for a power of $\alpha$:
 \begin{lemma}\label{lem:power}
If the theorem is true for $\alpha^p$ with $p\ge 2$, it is true for $\alpha$. 
 \end{lemma}
 \begin{proof} $G_{\alpha^p}=\langle \F_n,t^p\rangle$ is contained in
   $G_\alpha$ with finite index, so $\ker \varphi_{ |G_{\alpha^p}}$
   is contained in $\ker\varphi$ with finite index. 
   By   {\cite[Proposition~B1.11]{Str12}}, $[ \varphi_{ |G_{\alpha^p}}]\in\Sigma (G_{\alpha^p})$ if and only if   $[ \varphi ]\in\Sigma (G_{\alpha })$. 
   
   Applying the theorem to $\alpha^p$ yields elements $t_i$ in
   $G_{\alpha^p}\setminus \F_n$, hence  in $G_{\alpha }\setminus \F_n$.
These elements also work for $\alpha$. 
If $\ker \varphi_{
  |G_{\alpha^p}}$ virtually maps onto $\F_\infty$, so does
$\ker\varphi$. If $\ker \varphi_{ |G_{\alpha^p}}$ is free, so is
$\ker\varphi$ because it is torsion-free and  virtually free \cite{Swa69}.
 \end{proof}

By \fullref{lem:power} and \fullref{lem:upg} we may assume that $\Phi$ is UPG.
Take the hierarchy for $G$ provided by
\fullref{lem:upghierarchy}. 
Define $t_1,\dots,t_{n-1}$ as generators of the edge groups $A_i$ which occur  in the hierarchy.

\fullref{pasnul} says  that  
 $\Sigma(G)=\bigcap_{i\in\mathcal{I}} S(G,A_i)^c$,
so  $[\varphi]\in\Sigma(G)$ if and only if no $\varphi(t_i)$ equals 0. The $t_i$'s do not belong to $\F_n$ by 
the  ``moreover'' of \fullref{lem:upghierarchy}.

If $[\varphi]\in\Sigma(G)^c$ is discrete, \fullref{thm:Zhierarchy} 
says
$\ker\varphi$ surjects onto $\F_\infty$.

If $[\varphi]\in\Sigma(G)$, \fullref{thm:Zhierarchy} 
  says
$\ker\varphi$ is a free product whose factors are isomorphic to groups
$\ker\varphi_{|H_j}$, where $\{H_j\mid
j\in\mathcal{J}\}$ are the leaf groups of the hierarchy and $\ker\varphi _{|H_j}\ne H_j$. 
In this case, the leaf groups are $\Z^2$, so 
$\ker\varphi _{|H_j}$ is either $1$ or
$\Z$.
Thus, $\ker\varphi$ is a free group, and if $[\varphi]$ is discrete
then $\ker\varphi$ has finite rank.

This completes the 
proof of \fullref{main}.

 \begin{remark}
   Strebel \cite[Problem B1.13]{Str12} notes that $\Sigma$ is well
   behaved upon passing to finite index subgroups, and asks for
   examples in which calculating $\Sigma(G)$ directly is difficult,
   but $G$ contains a finite index subgroup $H$ for which $\Sigma(H)$
   can be computed.
Mapping tori of polynomially growing free group automorphisms provide
such examples.
 \end{remark}

\subsection{Examples}
\fullref{main} claims that 
$S(G)\setminus\Sigma(G)=\cup_{i}S(G,t_i)$
is the union of $n-1$ great subspheres. 
These subspheres are not necessarily distinct. 
For instance, 
the mapping torus of the trivial automorphism of $\F_n$ is isomorphic
to $\F_n\times \Z$; in this case the complement of $\Sigma(G)$ is a single sphere $S(G,t)$, with $t$ a generator of the $\Z$ factor. 
On the other hand, the following example shows that $n-1$
distinct spheres may be required.
\begin{example} 
Take a graph $\Gamma$ that is a circle with vertices $v_0,\dots,v_{n-1}$ and
edges $b_i=[v_{i-1}, v_i]$, with indices modulo $n$. 
At each vertex $v_i$ add a loop $a_i$.
Define a relative train track map on this graph fixing each $a_i$ and
sending each $b_i$ to $b_i a_i$.
This induces an outer automorphism $\Phi$ of the fundamental group $\F_n$ of
the graph, and we let $G=G_\Phi$. Choosing $v_i$ as a basepoint for $\Gamma$ defines a representative $\alpha_i\in\Aut(\F_n)$ of $\Phi$, and there are associated stable letters $t_i$ such that   $G =\F_n\rtimes_{\alpha_i}\Z=\langle \F_n,t \mid t_igt_i^{-1}={\alpha_i}(g)\rangle.$ These are the elements featured in \fullref{main}.

The mapping torus can be written as a graph of groups with underlying
graph a circle with edges corresponding to the $b_i$'s.
The vertex stabilizers are $\Z^2=\langle a_i,t_i\rangle$, because the $a_i$'s are fixed by
the automorphism.
The edge stabilizers are infinite cyclic. 
The edge corresponding to $b_i$ amalgamates $t_{i-1}$ to $t_ia^{-1}_i$. Since the images of 
$a_1,a_1a_2,\dots, a_1\cdots a_{n-2}$ in the abelianization of $G$ are
linearly independent, the spheres $S(G,t_i)$ are
distinct.  
\end{example}
Note the example is a linearly growing automorphism, so the number of spheres does not
correlate to degree of growth.
 
One might guess that the rank of the fiber blows up near
$\Sigma(G)^c$. 
The following example shows that this is not necessarily true.
\begin{example}\label{prod}
Consider $G=\F_n\times \Z=\langle x_1,\dots,x_n\rangle\times\langle
z\rangle$. 
Take positive, coprime integers $p$ and $q$,
and define $\varphi_{p,q}\from G\onto\Z$ by $\varphi_{p,q}(x_i)=p$ and
$\varphi_{p,q}(z)=q$.
Projection to the $\F_n$ factor is injective on $\ker\varphi_{p,q}$, and the
image is an index $q$ subgroup of $\F_n$ (consisting of all elements with exponent sum divisible by $q$), so 
$\ker\varphi_{p,q}$ has rank $q(n-1)+1$. 
It is a fiber of a fibration (whose monodromy has order $q$).
Now $[\varphi_{p,q}]=[\varphi_{1,\sfrac{q}{p}}]$, so
if we fix any $q$ and let $p$ grow, the sequence $([\varphi_{p,q}])_p\subset\Sigma(G)$
converges to $[\varphi_{1,0}]\in\Sigma(G)^c$ through fibrations with
constant fiber rank. 
\end{example}

\section{The rank of the fiber} \label{rank}

Let $G=G_\alpha=\F_n\rtimes_\alpha\Z$ be the mapping torus of a
polynomially growing automorphism  $\alpha\in\Aut(\F_n)$, with
$n\ge2$. 
By \fullref{main}, there exist elements $t_1,\dots, t_{n-1}$ in
$G\setminus \F_n $ (not necessarily distinct),
such that, given any surjection
$\varphi:G_\alpha\onto \Z$, either 
 some $\varphi(t_i)$ is $0$ and
  $\ker\varphi$ is infinitely generated, or   no $\varphi(t_i)$ is $0$   and $\ker\varphi$ is free of finite rank. 

  The main result of this section is the following.

\begin{theorem}\label{prop:fiberrank}
Let $\alpha\in\Aut(\F_n)$ be polynomially growing, and let $G_\alpha$ be its mapping
torus. There exist elements  $t_1,\dots,t_{n-1}\in G_\alpha\setminus \F_n$  as in \fullref{main} such that,
if $\varphi:G_\alpha\onto \Z$ is a surjection such that no $\varphi(t_i)=0$, the rank of the free group $\ker\varphi$ is
\[r=1+\frac1k\sum_{i=1}^{n-1} | \varphi(t_i) | \] 
with $k$   the least positive integer such that $\alpha^k\in\widehat\upg(\F_n)$.
\end{theorem}

 \begin{example} If $G=\F_n\times\Z$ is as in \fullref{prod}, one may take all $t_i$'s equal to $z$, and $r=1+(n-1) | \varphi(z) | $.   In particular, the rank of $\ker \varphi_{p,q}$ is $1+(n-1)q$. See \fullref{cen} for groups which are virtually $\F_n\times\Z$.
 \end{example}

\begin{corollary}  
Fibers with UPG monodromy have minimal rank.
\end{corollary}
\begin{proof}Suppose $\alpha\in\Aut(\F_n)$ is UPG.
  If $\varphi:G_\alpha\onto\Z$ has kernel  $\F_r$,
we have 
$$r=1+\sum_{i=1}^{n-1}|\varphi(t_i)|\geq n$$  since   $k=1$ and $|\varphi(t_i)|\geq 1$ for all $i$.  
\end{proof}

\begin{corollary}\label{corollary:unboundedfiberrank}
If $G$ is the mapping torus of a polynomially growing automorphism of
a non-abelian free group, and the first Betti number of $G$ is at least
2, then $G$ admits fibrations with fibers of unbounded rank.
\end{corollary}
\begin{proof}
Since the Betti number is $\ge2$, we may find $\varphi\from G\onto\Z$ with no $\varphi(t_i)$ equal to $0$,  and $\varphi(t_1)$ arbitrarily large.
\end{proof}

The rank of the fiber can alternatively be calculated as the degree of the Alexander polynomial of $G$
relative to $\varphi$ \cite{Mil68, But07}.
McMullen \cite{McM02} determined the Alexander invariants for
3--manifolds. 
Button \cite{But07} gives an algorithm for computing the relative Alexander polynomial
of a group admitting a deficiency 1 presentation, including mapping tori of free group
automorphisms, using Fox Calculus.
By generalizing McMullen's arguments, he is able to give lower bounds
for the Betti number of the kernel, in particular proving
\fullref{corollary:unboundedfiberrank} without the polynomial growth
hypothesis,   using the fact   
that the Alexander polynomial of the mapping
torus is non-constant
\cite[Theorem~3.4]{But07}. 

We prove \fullref{prop:fiberrank} without using the Alexander
invariants by computing the kernel 
from the hierarchical structure of the mapping torus. 

\begin{proof}[Proof of \fullref{prop:fiberrank}] We will consider restrictions of $\varphi$ to subgroups. Since they are not necessarily surjective, we
 rewrite the
   formula in a homogeneous way.
     
  For $\varphi$ real-valued but   with cyclic 
   image, and $H<G_\alpha$, define  $[G_\alpha:H]_\varphi$ as the index $[\varphi(G_\alpha):\varphi(H)]$. We   write $[G_\alpha:t_i]_\varphi$ rather than $[G_\alpha:\langle t_i\rangle]_\varphi$. With this notation, we have to prove:
  \[r-1=\frac1k\sum_{i=1}^{n-1} [G_\alpha:t_i]_ \varphi.\]
 
 First consider the $\upg$ case.
 As   in the proof of \fullref{main}, we argue by induction on the number of splittings in a hierarchy provided by \fullref{lem:upghierarchy}.
 In the proof of \fullref{thm:Zhierarchy} we showed that $K=\ker\varphi$
decomposes as the fundamental group of a graph of groups   $\Gamma_K$ with trivial edge groups.
The rank of $K$ is therefore the first Betti number  of the   graph $\Gamma_K$  plus the
sum of the ranks of the vertex groups.
The ranks of the vertex groups are computed inductively. 

The proposition is true if 
  the hierarchy is trivial. 
If not, we consider the first splitting in the hierarchy. We denote by $t_{1}$ a generator of the edge group of this splitting, and by $t_2,\dots, t_{n-1}$ generators associated to the other  edge groups of the hierarchy.

There are 2 cases.   First consider the HNN case 
  $G_\alpha=G_1*_{\langle t_{1}\rangle} $. 
By \fullref{lem:orbitcounting}, $\Gamma_K$ has $[G_\alpha:G_1]_\varphi$ vertices, each carrying a free group  of rank $1+\sum _{i=2}^{n-1}
 [G_1:t_i]_ \varphi$   by the induction hypothesis, and $ [G_\alpha:t_{1}]_ \varphi$ edges. 
The first Betti number of the graph is
$$1-[G_\alpha:G_1]_\varphi+[G_\alpha:t_{1}]_ \varphi,$$
so $$r=1-[G_\alpha:G_1]_\varphi+[G_\alpha:t_{1}]_
\varphi+[G_\alpha:G_1]_\varphi(1+\sum _{i=2}^{n-1}  [G_1:t_i]_
\varphi),$$ yielding  $$r=1+[G_\alpha:t_{1}]_ \varphi+\sum _{i=2}^{n-1}  [G_\alpha:t_i]_ \varphi
= 1+\sum _{i=1}^{n-1}  [G_\alpha:t_i]_ \varphi.$$
 
 The computation in the amalgam case is similar, except that there are two types of vertices. 
 
 If $\alpha$ is not UPG, let $G_k$ denote $G_{\alpha^k}$, which is an index $k$ subgroup of
$G_\alpha$.
 Let $\varphi_k$ be  the restriction of $\varphi$ to $G_k$, and let
 $r_k$ the rank of $\ker\varphi_k$. 
We may take the same $t_i$'s for $\alpha$ and $\alpha^k$ (see the proof of \fullref{lem:power}),
and
\[[\ker\varphi:\ker\varphi_k](r-1)=r_k-1=\sum_{i=1}^{n-1}[G_k:t_i]_\varphi.\]
  Considering the exact sequence $1\to\ker\varphi\to G_\alpha\to\R$ and restricting to 
  $G_k$, which  has index $k$ in $G_\alpha$, we see that   $$k=[G_\alpha:G_k]_\varphi\ [\ker\varphi:\ker\varphi_k],$$ 
   so multiplying by  $[G_\alpha:G_k]_\varphi$ yields:  
 \[k (r-1)= [G_\alpha:G_k]_\varphi\sum_{i=1}^{n-1}[G_k:t_i]_\varphi= \sum_{i=1}^{n-1}[G_\alpha:t_i]_\varphi.
  \qedhere\] 
 \end{proof}

\section{Finite order automorphisms and $\gbs$ groups with center}
\label{cen}

We now suppose that $G$ is the mapping torus of a finite order element
$\Phi\in\Out(\F_n)$, for some $n\ge2$. 
By \cite[Proposition~4.1]{Lev13gbsrank} this is equivalent to $G$ being a GBS group with non-trivial
center, and being non-elementary  (i.e.\   not  isomorphic to $\Z$,
$\Z^2$, or the Klein bottle group). 
It has a finite index subgroup isomorphic to $\F_m\times\Z$ for some
$m\ge 2$. 
The main results of this section (\fullref {lem:fiberrankfor0growth} and \fullref {lem:fiberrankfor0growth2}) will give quantitative versions of these facts.

We refer to \cite{For03, For06, Lev07,Lev13gbsrank} for basic facts about GBS groups.  In terms of the modular map $\Delta:G\to\Q^*$ of \fullref{sec:gbs}, a non-elementary GBS group has non-trivial center if and only if $\Delta$ is trivial.

The group $G$ acts on a tree $T$ with infinite cyclic edge and vertex stabilizers (as usual, we assume that the action of $G$ on $T$ is minimal). 
Elements of $G$ fixing a point in $T$ are called elliptic. The set of elliptic elements does not depend on the choice of $T$, and consists of finitely many conjugacy classes of cyclic subgroups (because $T/G$ is a finite graph).

The  center of $G$, denoted by $Z$, is   infinite cyclic and equals the set of
elements acting as the identity on $T$ \cite[Proposition~2.5]{Lev07}.  
This implies that, if a subgroup $H<G$ acts on $T$ minimally, in particular if $H$ has finite index, or is normal and   non-central, its centralizer is equal to $Z$.

The quotient $G/Z$ acts on $T$ with finite stabilizers, so is virtually free. If $L<G/Z$ is free, its preimage is  isomorphic to $L\times \Z$ because it is a central $\Z$-by-free extension.

 \begin{definition} Given a non-elementary GBS group $G$ with non-trivial center $Z$, we define numbers $\kappa$ and $\epsilon$ as follows:
 \begin{itemize}
 \item $\kappa$ is the lcm of orders of torsion elements of $G/Z$;
 \item $\varepsilon+1$ is the smallest rank of a    free subgroup of finite index   $L< G/Z$.
 \end{itemize}
 \end{definition}

  The group $G/Z$ is the fundamental group of a finite graph of groups whose vertex groups are finite cyclic groups, and $\kappa$ is the lcm of  their orders. 
 If   $L \subset G/Z$ is a free subgroup of finite index, it is well-known that its index is divisible by $\kappa$ (each vertex group 
   acts freely on the set of cosets modulo  $ L $),    and \emph{$\kappa$ is the smallest index of a free subgroup of $G/Z$}
  (Serre \cite[II.2.6, Lemma 10]{Ser03}   defines a map of $G/Z$ into permutations of $\kappa$
elements so that the vertex groups act freely;    
  the
preimage of any point stabilizer is a\ free subgroup of index $\kappa$).

It follows that 
there exist free subgroups of index $p\kappa$ for every $p\ge1$. An Euler characteristic argument shows that the rank  is $p\varepsilon+1$ if the index is $p\kappa$.

 Going back to $G$, we may view  $\kappa$ as the smallest integer such that $a^\kappa\in Z$ for every elliptic element $a\in G$. We also have:
 
 \begin{lemma} \label{nel}
 Let $E\subset G$ be the   (normal) subgroup generated by all elliptic elements. If $\varphi:G\to\R$ is a homomorphism such that $\varphi(Z)\ne0$, then $\kappa=[\varphi(E):\varphi(Z)]$.
 \end{lemma}
 
 Since any elliptic element has a power in $Z$, the restriction of $\varphi$ to $E$ is unique up to a multiplicative constant.
 
 \begin{proof} Let $z$ be a generator of $Z$. It acts as the identity on $T$, so if $a$ generates a vertex stabilizer there exists $\kappa_a$ such that $z=a^{\kappa_a}$. The number $\kappa$ is the lcm of the $\kappa_a$'s, while $\varphi(E)$ is generated by the numbers $\varphi(z)/\kappa_a$. The lemma follows.
 \end{proof}
 
  \begin{theorem}\label{lem:fiberrankfor0growth}
  Let $G$ be a non-elementary GBS group $G$ with non-trivial center   $Z$. Given positive integers $k$ and $n$, the following are equivalent:
\begin{enumerate}
\item $(k,n-1)$ is an integral  multiple of $(\kappa,\epsilon)$;
\item   $G/Z$ has a subgroup of index $k$ isomorphic to $\F_n$;
\item $G$ has a 
subgroup $G_0$ of index $k$ isomorphic to $\F_n\times\Z$, with $Z\subset G_0$; 
\item $G$ has a subgroup $G_0$ isomorphic to $\F_n\times\Z$ whose index is finite and equal to $k\, [Z:G_0\cap Z]$.
\end{enumerate}
 \end{theorem}

 \begin{proof}
We already know that (1) and (2) are equivalent. For $p\ge1$, let $L\simeq \F_{p\varepsilon+1}$ have index $p\kappa$ in $G/Z$.
 Its preimage in $G$ contains $Z$, has index $p\kappa$, and is isomorphic to $\F_{p\varepsilon+1}\times\Z$. This proves that (1) implies (3),
 and (3) trivially implies (4). Conversely, if $G_0$ is as in (4), the $\Z$ factor is contained in $Z$  
   (it centralizes $G_0$, which acts minimally on $T$, so it acts
  trivially on $T$).  The image of $G_0$ in $G/Z$ is free of rank $n$ and has index $k=[G:G_0]/[Z:G_0\cap Z]$, so (2) holds.
 \end{proof}

  We now consider fibrations of $G$.

 \begin{theorem}\label{lem:fiberrankfor0growth2}
  Let $G$ be a non-elementary GBS group $G$ with non-trivial center. Let $k$ and $n$ be positive integers.
\begin{enumerate}
\item   If the first Betti number of $G$ is 1, 
there exists an element $\Phi\in\Out(\F_n)$ of order $k$ such that $G\simeq G_\Phi$
if and only if $(k,n-1)= (\kappa,\epsilon)$. 
\item If the first Betti number of $G$ is at least 2, there exists an element $\Phi\in\Out(\F_n)$ of order $k$ such that $G\simeq G_\Phi$
if and only if  $(k,n-1)$ is an integral  multiple of $(\kappa,\epsilon)$. 
\end{enumerate}
 \end{theorem}

The first Betti number $b_1(G)$ of $G$ is equal to $1+b_1(\Gamma)$, with $\Gamma$ the quotient graph $T/G$ (see \cite[Proposition~3.3]{Lev07}),   and $G/E$ is free of rank $b_1(\Gamma)$ (this is a general fact about graphs of groups).
In particular,  $b_1(G)=
1$ is equivalent to  $\Gamma $ being  a tree, and to $E=G$.

 \begin{corollary}  
If the first Betti number of $G$  is at least 2, ranks of fibers are an arithmetic progression: there is an exact sequence $1\to\F_n\to G\to\Z\to1$ if and only if $n$ is of the form $p\varepsilon+1$ with $p\ge1$.
 \qed
\end{corollary}

Before proving the theorem, we note:

\begin{lemma} \label{ordre}
If $G=G_\Phi$, with $\Phi\in\Out(\F_n)$ of   order $k$, and $\varphi:G\onto\Z$ is the associated fibration, then $\varphi(Z)=k\Z$    and $\varphi^{-1}(k\Z)\simeq \F_n\times\Z$. 
\end{lemma}

\begin{proof} Let $\alpha\in\Aut(\F_n)$ be a representative of $\Phi$, let $t$  be the associated stable letter, and let $z$ be a generator of $Z$.   Writing $z=gt^q$ with $g\in\F_n$, we   show $ | q| =| k| $. Since $z$ centralizes $\F_n$, the automorphism $\alpha^q$ is inner, so $q$ is a multiple of $k$.
 On the other hand,  since $\Phi$ has order $k$, some $ht^k$, with $h\in \F_n$,  centralizes $\F_n$ hence belongs to $Z$ (because $\F_n$ is normal so acts minimally on $T$).      This implies that $k$ is a multiple of $q$, and therefore $ | q| =| k| $. 
The extension $1\to\F_n\to\varphi^{-1}(k\Z)\to k\Z\to1$ is trivial because $\varphi^{-1}(k\Z)$ contains $z^{\pm1}$.
\end{proof}

\begin{proof}[Proof of \fullref{lem:fiberrankfor0growth2}]

  Suppose $G=G_\Phi$ with $\Phi\in\Out(\F_n)$ of   order $k$. 
\fullref{ordre} shows that  $\varphi^{-1}(k\Z)$ is a subgroup of index $k$ isomorphic to $\F_n\times\Z$ and containing $Z$, so  
$(k,n-1)$ is a multiple of $(\kappa,\varepsilon)$    
by \fullref{lem:fiberrankfor0growth}. 
If $b_1(G)=1$, we have seen that $E=G$, so $k=\kappa$ by  \fullref{nel}. We have   proved the ``only if'' direction in both assertions of the theorem.

 For the converse, first recall how to construct maps $\varphi:G\to\R$ with $\varphi(Z)\ne0$ (see \cite[Proposition~3.3]{Lev07}). View $G$ as the fundamental group of a graph of infinite cyclic groups. Consider a standard generating set, consisting of generators of vertex groups and stable letters; the number of stable letters is the first Betti number of the graph $\Gamma=T/G$,   equal to $b_1(G)-1$. One first defines $\varphi$ on $E$ (as mentioned above,   $\varphi_{ | E  }$ is unique up to scaling). One must then choose the image of the  stable letters. Non-triviality of the center ensures that any choice yields a well-defined map $\varphi:G\to\R$.

   If $b_1(G)\ge2$, there is at least one stable letter and one may construct an epimorphism $\varphi_p:G\onto \Z$ such that $\varphi_p(E)=p\Z$ for any $p\ge1$. If $b_1(G)=1$  there is no stable letter, $E=G$, and  we only obtain $\varphi_1$.

We   have $\varphi_p(Z)=p\kappa\Z$ by   \fullref{nel}, and we let $G_p$ be  the preimage   $\varphi_p^{-1}(p\kappa\Z)$. It is a normal subgroup of index $p\kappa$ containing $Z$ and $\ker\varphi_p$.

Consider the projection $\pi:G\to G/Z$. The groups $G_p$ and $\ker\varphi_p$ have the same image, and $\pi$ is injective on $\ker\varphi_p$. 
Since every elliptic element    $a\in G_p$ belongs to $Z$ (because $\langle a,Z\rangle$ is cyclic), the image of $G_p$ in $G/Z$ is torsion-free. It is a free subgroup of index $p\kappa$, and its rank is $p\varepsilon+1$.
 We therefore have $\ker\varphi_p\simeq \F_{p\varepsilon+1}$. By \fullref{ordre}, the associated monodromy $\Phi$ has order $p\kappa$. 
\end{proof}

\begin{remark} The proof shows that $G$ has a \emph{normal} subgroup   of index $\kappa$ isomorphic to $\F_n\times\Z$ and containing  $Z$, and $G/Z$ has a \emph{normal} free subgroup of index $\kappa$.
\end{remark}

One might ask whether the monodromies of fibrations with fiber of
minimal rank $\varepsilon+1$ for a fixed $\gbs$ group $G$ must be conjugate in 
$\Out(\F_{\varepsilon+1})$, or whether this must be true for fibrations in
the same component of $\Sigma(G)$.
The answer to both questions is `no' once $n=\varepsilon+1\geq 4$.

For $n=2$, Bogopolski, Martino, and Ventura \cite{BogMarVen07} show that mapping tori of
$[\alpha_1],\,[\alpha_2]\in\Out(\F_2)$ are isomorphic if and only if
$[\alpha_1]$ is conjugate to $[\alpha_2]$ or $[\alpha_2]^{-1}$ in $\Out(\F_2)$.

An example of Vikent'ev \cite{Vik07} shows this statement is not true in general
for $n=3$, but Khramtsov \cite{Khr90} showed it is true for finite
order outer automorphisms.

For $n=4$, Khramtsov \cite{Khr90} gives the following example.
Let $G$ be the $\gbs$ group $\langle a,b,t\mid a^4=b^2,\, [b,t]=1\rangle$.
Consider $\varphi_i\from G\onto \Z$ defined by $\varphi_1(a)=1$,
$\varphi_1(b)=2$, $\varphi_1(t)=0$ and $\varphi_2(a)=1$,
$\varphi_2(b)=2$ and $\varphi(t)=1$.
These surjections give fibrations of $G$ with fiber of rank 4 and
monodromy of order 4 in $\Out(\F_4)$.  
The minimal fiber rank for $G$ is $n=\varepsilon+1=4$, by
\fullref{lem:fiberrankfor0growth2}, since $k=\kappa=4$.
In this example, $\Sigma(G)\cong S^1\setminus S^0$ consists of classes of homomorphisms not killing the element $a$.
Thus, $[\varphi_1]$ and $[\varphi_2]$ are in the same component of
$\Sigma(G)$, since they both send $a$ to a positive number.
Khramtsov gives an ad hoc argument to show the monodromies are
non-conjugate. 
This can also be verified using a solution to the conjugacy problem
for finite order elements of $\Out(\F_n)$.
Such solutions follow from work of 
Krsti\'c \cite{Krs89},  and are explained in
\cite{Khr95} or \cite{KrsLusVog01}.


\bibliographystyle{hyperamsplain}
\bibliography{fibrations}

\end{document}